\documentclass[12pt]{amsart}
\usepackage{amsmath,amssymb,amsthm, mathrsfs}
\usepackage{mathtools}
\usepackage{stmaryrd}
\usepackage{graphicx,setspace}
\usepackage{tikz}
\usepackage{tikz-cd}
\usetikzlibrary{arrows}
\usepackage{wasysym}
\usepackage[makeroom]{cancel}
\usepackage{fullpage}
\usepackage{comment}
\usepackage{enumitem}
\usepackage{caption}
\usepackage{cancel}
\usepackage[OT2,T1]{fontenc}
\DeclareSymbolFont{cyrletters}{OT2}{wncyr}{m}{n}
\DeclareMathSymbol{\Sha}{\mathalpha}{cyrletters}{"58}
\allowdisplaybreaks

\theoremstyle{plain}
\newtheorem{theorem}{Theorem}[section]
\newtheorem{theorem*}{Theorem}
\newtheorem{proposition}[theorem]{Proposition}
\newtheorem*{proposition*}{Proposition}

\newtheorem{lemma}[theorem]{Lemma}
\newtheorem{corollary}[theorem]{Corollary}

\newtheorem{question}[theorem]{Question}

\newtheorem*{claim*}{Claim}
\theoremstyle{definition}
\newtheorem{example}[theorem]{Example}
\newtheorem{remark}[theorem]{Remark}
\newtheorem{definition}[theorem]{Definition}
\newtheorem*{definition*}{Definition}

\def\F{\mathbb{F}}

\def\Z{\mathbb{Z}}

\def\P{\mathbb{P}}

\date{\today}


\title{Lifting Elementary Abelian Covers of Curves}
\author{Jianing Yang}

\begin{document}

\maketitle

\begin{abstract}
Given a Galois cover of curves $f$ over a field of characteristic $p$, the lifting problem asks whether there exists a Galois cover over a complete mixed characteristic discrete valuation ring whose reduction is $f$. In this paper, we consider the case where the Galois groups are elementary abelian $p$-groups. We prove a combinatorial criterion for lifting an elementary abelian $p$-cover, dependent on the branch loci of lifts of its $p$-cyclic subcovers. We also study how branch points of a lift coalesce on the special fiber. Finally, we analyze lifts for several families of $(\Z/2)^3$-covers of various conductor types, both with equidistant branch locus geometry and non-equidistant branch locus geometry.
\end{abstract}


\section{INTRODUCTION}
Given a smooth curve over a field $k$ of characteristic $p$, we can study its lift to characteristic \nolinebreak$0$, which is a smooth (relative) curve over a mixed characteristic complete discrete valuation ring $R$ with residue field $k$. Moreover, if we let a finite group act on the curve in characteristic $p$ and take the quotient, we obtain a Galois cover of such curves. The Lifting Problem asks: given a Galois cover of smooth curves in characteristic $p$, $X\xrightarrow{G}\P^1_k$, when can we lift it to characteristic $0$? Which groups can be realized as Galois groups of covers that lift? One famous result in the area is the Oort conjecture, which states that all cyclic covers lift. This topic is also related to the Inverse Galois Problem, deformation theory, étale fundamental groups, and patching.

The focus of this paper is on the elementary abelian case, i.e., on $(\Z/p)^n$-covers of smooth projective curves. It is known that some of them lift, while some of them do not (see Example \ref{eleabe}), but results about when they lift are very incomplete. The main result of this paper, which generalizes Barry Green and Michel Matignon's criterion for lifting $\Z/p\times\Z/p$-covers \cite{GM}, applies to all elementary abelian $p$-covers of $\P^1_k$, where \nolinebreak$k$ is an algebraically closed field of characteristic $p$. I show the following branch cycle criterion, a precise version of which will be stated in Section $3$ (Theorem \ref{branchcyclecriterion}).

\begin{theorem}[Imprecise version]
Let $C: X\to\P^1_k$ be a $(\Z/p)^n$-Galois cover, and $m_1+1\leq \cdots\leq m_n+1$ be the conductors of its $n$ generating $\Z/p$-subcovers. Then $C$ can be lifted to characteristic \nolinebreak$0$ if and only if $m_i\equiv -1$ mod $p^{n-i}$ for $1\leq i\leq n-1$ and these $\Z/p$-subcovers can be respectively lifted with branch loci $B_1,\ldots,B_{n}$ that satisfy a certain combinatorial criterion.
\end{theorem}

Moreover, I relate the $p$-rank stratification of the Artin-Schreier space to a stratification of the characteristic $0$ Hurwitz space by the branch locus coalescing behavior of $p$-cyclic covers in characteristic $0$ (Section \ref{coalesce}). I also classify all admissible Hurwitz trees for certain types of $(\Z/2)^3$-covers (Section \ref{444}). Finally, I construct explicit lifts for a new family of $(\Z/2)^3$-covers, with non-equidistant geometry (Section \ref{442r}), providing the first example with non-constant conductor type for this group.

This paper is adapted from my dissertation at University of Pennsylvania. I would like to thank my advisor David Harbater for his guidance.

\section{The Lifting Problem and Oort Groups}
Throughout this paper, we let $k$ be an algebraically closed field of characteristic $p$, $R$ be a finite extension of $W(k)$, the ring of Witt vectors over $k$, and $K$ be the fraction field of $R$. We always allow finite extensions of $R$ if necessary. Let $\pi$ be a uniformizer of $R$, and $v$ be the valuation on $R$ with respect to $\pi$. A \textit{curve} is assumed to be reduced, smooth, connected, and projective unless stated otherwise. A $G$-(Galois) \textit{cover of curves} $X\to Y$ is a finite, generically separable morphism such that the group of automorphisms $\mathrm{Aut}_Y(X)$ is isomorphic to $G$, and acts transitively on each geometric fiber.

\subsection{The Global Lifting Problem}
We can state the (global) lifting problem as follows:

\begin{question}[\textbf{The global lifting problem}]
Let $f: X_k\xrightarrow{G} \P^1_k$ be a Galois branched cover of smooth projective curves. Does there exist some $R$ as above,  together with a Galois cover $X_{R}\xrightarrow{G}\P^1_{R}$ of smooth projective $R$-curves whose special fiber is $f$? If the answer is yes, we say that $f$ \textit{lifts}.
\end{question}

\begin{remark}
A smooth projective curve always lifts over any complete discrete valuation ring $R$ with residue field $k$ \cite[III, Corollaire 6.10 and Proposition 7.2]{SGA03}. However, simply taking the equation defining $X_k$, lifting its coefficients to $R$ does not always work, since there may not be a $G$-action on $X_R$ that reduces to the one on $X_k$.
\end{remark}

There are various obstructions to lifting. For example, the Hurwitz bound in characteristic $0$ \cite[IV.2]{hartshorne} tells us that, if $|G|>84(g(X)-1)$, then $X_k\xrightarrow{G} \P^1_k$ does not lift.
A key statement concerning the lifting problem is the Oort conjecture.

\begin{theorem}[\textbf{Oort conjecture}]\label{oortconj}
The lifting problem has a solution if $G$ is cyclic.
\end{theorem}

The proof reduces to the case $\Z/p^n$, and was proven in a series of papers. In the case of prime to $p$ groups, it was proven by Grothendieck in \cite[Exp. XIII]{SGA03}, using the ``tame Riemann existence converse'' (see \cite[Theorem 1.5]{obus17}). The $\Z/p$ case was proven by Oort-Sekiguchi-Suwa \cite{OSS} in 1989, using Artin-Schreier theory. The $\Z/p^2$ case was proven by Green-Matignon \cite{GM} in 1998, by reducing to the local lifting problem and using Artin-Schreier-Witt theory. Finally the Oort conjecture was proven for general cyclic groups by Obus-Wewers \cite{OW} and Pop \cite{Pop} in 2014.

This result motivates the natural question: for which finite groups $G$ do \textit{all} $G$-covers lift? For which finite groups $G$ do \textit{some} $G$-covers lift? we define the following:

\begin{definition}[\cite{CGH}]
A finite group $G$ for which every $G$-Galois cover $X\to\P^1_k$ lifts to characteristic $0$ is called an \textit{Oort group} for $k$. If there exists a $G$-Galois cover that lifts, $G$ is called a \textit{weak Oort group}.
\end{definition}

In particular, all Oort groups are weak Oort groups. The Oort conjecture states that cyclic groups are Oort groups.

\begin{example}\label{eleabe}
Let $X_k=\P^1_k$, and $G=(\Z/p)^n$. Then $G$ embeds into the additive group of $k$ and has an additive action on $X_k$. Suppose that the $G$-Galois cover $X_k\to \P^1_k$ lifts to $R$. Then $G$ acts on the generic fiber $X_K$. However, since the genus of $X_K$ is $0$, the group of automorphisms of $X_K$ embeds into $\mathrm{PGL}_2(\bar{K})$, which does not contain $(\Z/p)^n$ for $n>1$ except for $(\Z/2)^2$. Therefore, elementary abelian $p$-groups, apart from $p$-cyclic groups and the Klein-four group, are not Oort groups. Meanwhile, they are shown to be be weak Oort groups in \cite{mat}.
\end{example}

\subsection{The Local Lifting Problem}
A local-global principle \cite{garuti} reduces the lifting problem to one of local nature.

\begin{question}[\textbf{The local lifting problem}]
Suppose $G$ is a finite group, and $k[[z]]/k[[t]]$ is a (possibly ramified) $G$-Galois extension. Does there exist some $R$, and a $G$-Galois extension $R[[Z]]/R[[T]]$ such that the $G$ action on $R[[Z]]$ reduces to the given $G$ action on $k[[z]]$?
\end{question}

\begin{definition}\label{locallifts}
If the local lifting problem has a solution for a $G$-extension $k[[z]]/k[[t]]$, we say that the extension \textit{lifts to characteristic $0$}, and $R[[Z]]/R[[T]]$ is a \em{lift of the extension}.
\end{definition}

By abuse of terminology, we will say $R[[Z]]/R[[T]]$ is a $G$-\textit{cover}. For such local covers, we have the following notion of (geometric) branch points.

\begin{definition}\label{branchlocus}
Let $R[[Z]]/R[[T]]$ be a $G$-extension. Assume that the cover $f:\mathrm{Spec}R[[Z]]\to\mathrm{Spec}R[[T]]$ is unramified at the prime ideal $(\pi)$. Then the \textit{branch points} of $R[[Z]]/R[[T]]$ are the divisors $b$ of $\mathrm{Spec}R[[T]]$ such that $f$ is ramified at $f^{-1}(b)$. Enlarge $R$ so that all branch points are $R$-rational. The set of branch points is called the \textit{branch locus} of $R[[Z]]/R[[T]]$.
\end{definition}

We then have the corresponding definitions for local Oort groups and weak local Oort groups for $k$, and from now on when we say (weak) Oort groups, we mean (weak) local Oort groups.

\begin{definition}
    A cyclic-by-$p$ group $G$ for which every $G$-extensions $k[[z]]/k[[t]]$ lifts to characteristic $0$ is called a \textit{local Oort group} for $k$. If there exists a local $G$-extension that lifts to characteristic $0$, $G$ is called a \textit{weak local Oort group}.
\end{definition}

The Galois group of any such local extension is a finite cyclic-by-$p$ group. Obstructions due to \cite{CGH}\cite{BW} give that all local Oort groups must be one of the following: cyclic groups, dihedral groups $D_{p^n}$ for any $n$, and the group $A_4$ (for $\mathrm{char}(k) = 2$).

All these possible candidates are known to be Oort groups, apart from dihedral groups of higher orders (\cite{BW06}, \cite{pagot}, \cite{obus15}, \cite{obus16}, \cite{weaver}, \cite{dang}).

Meanwhile, the question of whether a finite group $G$ is a weak Oort groups is sometimes called the Inverse Galois Problem for lifting.

For a weak Oort group $G$ that is not an Oort group, we can look more closely and ask which $G$-covers lift. The subjects of this study here are the elementary abelian $p$-covers in particular. Matignon, in proving elementary abelian $p$-groups are weak Oort groups, constructs lifts for a special family of $(\Z/p)^n$-covers of type $(p^{n-1},\ldots,p^{n-1})$.

\begin{theorem}[\cite{mat}]
$(\Z/p)^n$ is a weak Oort group for all $n\geq 1$. 
\end{theorem}

The lifts constructed in \cite{mat} all have equidistant geometry, which necessitates that the covers have constant conductor type.

In Theorem 4.3.4 of Pagot's thesis \cite{pagot thesis}, he constructs lifts for families of $(\Z/p)^n$-covers, which have non-equidistant geometry, as implied by Lemme 4.1.1 of the thesis. In Section \ref{442r}, we construct lifts for a new family of $(\Z/2)^3$-covers, with non-constant conductor types (unlike Pagot's examples for that group), and also with non-equidistant geometry. Since the posting of this manuscript, Pagot has informed me that he can now obtain further results about lifting $(\Z/2)^3$-covers by using Theorem \ref{branchcyclecriterion} below. Those results will appear in a forthcoming paper of his.

\section{Branch Cycle Criterion for $(\Z/p)^n$-Covers}

In this section, we first prove lemmas on the degree of the \textit{special different}, i.e.\ the degree of the different of $k((z))/k((t))$, related to ramification jumps, and the degree of the \textit{generic different}, i.e.\ the degree of the different of $K((Z))/K((T))$. Then we arrive at the main result of this paper (Theorem \ref{branchcyclecriterion}), which is a combinatorial criterion for lifting elementary abelian $p$-covers.

\subsection{Ramification Jumps}\label{ramjumpsection}
\begin{definition}\label{conductor}
For a $\Z/p$-extension $k((z))/k((t))$ given by the Artin-Schreier equation $$z^p-z=f(\frac{1}{t}),$$ where $f(\frac{1}{t})\in k[t^{-1}]$, we call $m+1:=\mathrm{deg}(f)+1$ the \textit{conductor} of the extension.
\end{definition}

Since $k$ is algebraically closed, after a change of variable, we can assume $f(\frac{1}{t})=\frac{1}{t^m}$.

\begin{remark}
All such $\Z/p$-extensions are defined by an equation of the above form. By the Oort Conjecture, it lifts to a cover $R[[Z]]/R[[T]]$, with number of (geometric) branch points equal to the conductor.
\end{remark}

\begin{definition}
Let $L/K$ be a $G=(\Z/p)^n$-Galois totally ramified extension of local fields in characteristic $p$. Let $I_l$ (resp.\ $I^l$) be the $l$-th ramification group in lower numbering (resp.\ upper numbering). For $0\leq i\leq n-1$, define the \textit{$i$-th lower (upper) ramification jump} be the positive integer $l$ such that the $p$-rank of $I_l$ (resp.\ $I^l$) is at least $n-i$ and the $p$-rank of $I_{l+1}$ (resp.\ $I^{l+1}$) is at most $n-i-1$.
\end{definition}

In this section, the ramification jumps are always with respect to the lower numbering unless specified otherwise. Note that the ramification jumps can coincide when the quotient $I_l/I_{l+1}$ has order greater than $p$.

 \begin{lemma}\label{ramjumps}
 Let $L/K$ be a $G=(\Z/p)^n$-Galois totally ramified extension of complete discretely valued fields with residue characteristic $p$. Suppose $L/K$ can be written as a tower of $\Z/p$-extensions $L=K_n/K_{n-1}/\ldots/K_1/K_0=K$, where $K_{i+1}/K_i$ has conductor $m^{(i)}+1$, such that $m^{(0)}\leq m^{(1)}\leq \cdots \leq m^{(n-1)}$. Then the $l$-th lower ramification jump of $L/K$ is $m^{(l)}$. Moreover, the degree of the different of $L/K$, as in \cite[IV.2, Proposition 4]{serre}, is $\displaystyle\sum_{l=0}^{n-1}(m^{(l)}+1) p^{n-l-1}(p-1)$.
 \end{lemma}
 
 \begin{proof}
First we use induction on $n$ to compute the ramification jumps. When $n=1$, $G=\Z/p$. Let $m+1$ be the conductor of $L=K_1/K_0=K$. Then by \cite{serre}, Chapter IV, exercise 2.5, $G_m=\Z/p$ and $G_{m+1}=1$. Thus the unique ramification jump is one less than the conductor.
 
 Suppose the statement about the ramification jumps is true for $n-1$. Consider $L/K$ as in the hypothesis, and let $H=(\Z/p)^{n-1}$ be the Galois group of $L/K_1$. Then $m^{(l)}$, for $1\leq l\leq n-1$, is the $(l-1)$-th ramification jump of $H$. Let $H_i$ be the $i$-th ramification group of $L/K_1$. First note that by \cite[page 73]{serre}, $$\varphi_{L/K_1}(m^{(0)})+1=\frac{1}{|H_0|}\sum_{i=0}^{m^{(0)}}|H_i|=\frac{1}{p^{n-1}}(m^{(0)}+1)p^{n-1}=m^{(0)}+1,$$ where $\varphi$ is the Herbrand function \cite[IV.3]{serre}, and $\varphi_{L/K_1}(m)>m^{(0)}$ for $m>m^{(0)}$. Since $m^{(0)}\leq m^{(1)}$, $H=H_{m^{(0)}}=I_{m^{(0)}}\cap H$ by \cite[Chapter IV, Proposition 2]{serre}, and $H\subseteq I_{m^{(0)}}$. Thus $I_{m^{(0)}}/H=I_{m^{(0)}}H/H=(G/H)_{\varphi_{L/K_1}(m^{(0)})}=(G/H)_{m^{(0)}}=\Z/p$ by Herbrand's theorem \cite[IV.3, Lemma 5]{serre}, hence $I_{m^{(0)}}=(\Z/p)^n$.
 
 Now, let $l$ be the largest integer such that $m^{(l)}=m^{(0)}$. We have $$I_{m^{(l)}+1}H/H=(G/H)_{\phi_{L/K_1}(m^{(l)}+1)}=1,$$ so $I_{m^{(l)}+1}\subseteq H$. Then $I_{m^{(l)}+1}=I_{m^{(l)}+1}\cap H=H_{m^{(l)}+1}=(\Z/p)^{n-l-1}$. Therefore $m^{(i)}$ is the $i$-th ramification jump of $L/K$ for all $0\leq i\leq l$. Similarly, for all $i>l$, $I_{m^{(i)}}=I_{m^{(i)}}\cap H=H_{m^{(i)}}$, and $I_{m^{(i)}+1}=I_{m^{(i)}+1}\cap H=H_{m^{(i)}+1}$, so $m^{(i)}$ is the $i$-th lower ramification jump.
 
 Finally, by \cite[IV.2, Proposition 4]{serre}, we get that the degree of the different of $L/K$ is
 \begin{align*}
 d_s=\sum_{j=0}^\infty (|I_j|-1)=\sum_{l=0}^{n-1}(m^{(l)}-m^{(l+1)})(p^{n-l}-1)=\sum_{l=0}^{n-1}(m^{(l)}+1) p^{n-l-1}(p-1).
 \end{align*}
 \end{proof}
 
\begin{remark}
For $L/K$ as in Lemma \ref{ramjumps}, if we take the tower of extensions $L/L^{G_{j_{n-1}}}/\cdots/ L^{G_{j_0}}$, where $j_i$ is the $i$-th lower ramification jump, then the associated sequence of conductors is ascending.
\end{remark}

\subsection{Conductor Type}\label{condtype}
For an elementary abelian $p$-cover $k[[z]]/k[[t]]$ where $G=(\Z/p)^n$, whether a $G$ can be lifted to characteristic $0$ often depends on the conductors of its $\Z/p$-subcovers. For ease of notation, we define a $(\Z/p)^n$-cover of certain (conductor) type.

\begin{definition}\label{type}
Let $G=(\Z/p)^n$, and assume that $G$ is a group of automorphisms of $k[[z]]$ as a $k$-algebra. Suppose that $(m_1,\ldots,m_n)$ is the lexicographically smallest $n$-tuple of integers such that there exists subgroups $G_1,\ldots,G_n\subset G$ of index $p$ that satisfy the following conditions:
\begin{enumerate}
\item The $p$-cyclic extensions $k[[z]]^{G_i}/k[[z]]^G$ have conductors $m_i+1$,
\item $k[[z]]^{G_1},\ldots,k[[z]]^{G_n}$ are linearly disjoint over $k[[z]]^G$.
\end{enumerate}
Then we say that $k[[z]]/k[[z]]^G$ is a \textit{cover of type} $(m_1+1,\ldots,m_n+1)$, with respect to $G_1,
\ldots, G_n$. Note that $m_1\leq m_2\leq\cdots\leq m_n$.
\end{definition}

\begin{remark}
    Note that this is different from the notations in Mitchell's thesis \cite{mitchell}, where he calls such covers of type $(m_1,\ldots,m_n)$.
\end{remark}

\begin{proposition}\label{naseq}
With the notations in the above definition, let $K_0:=k((t))$ and $K_i=k((z))^{G_i}$ for $1\leq i\leq n$. Then after a change of variable, the $\Z/p$-extensions $K_1$ and $K_i$ (for $i\geq 2$) over $K_0$ are simultaneously defined by Artin-Schreier equations:
\begin{align*}
w_1^p-w_1&=f_1(\frac{1}{t})=\frac{1}{t^{m_1}}\\
w_i^p-w_i&=f_i(\frac{1}{t})=\sum_{1\leq j\leq m_i,\ p\nmid m'}\frac{c_{i,j}}{t^j},
\end{align*}
where for $i\neq j$, the leading coefficients of $f_i$ and $f_j$ are $\F_p$-linearly independent if $m_i=m_j$, and where $c_{i,m_i}\not\in \F_p$ for $i\geq 2$.
\end{proposition}

\begin{proof}
First, for some uniformizer $t$ of $K_0$, $K_1/K_0$ can be defined by $w_1^p-w_1=\frac{1}{t^{m_1}}$, and with that same uniformizer $t$, $K_i/K_0$ can be defined by Artin-Schreier equations as above.

Suppose that $m_i=m_j$ for some $i<j$. Then $(m_i,m_j)$ must also be the lexicographically the smallest tuple of conductors satisfying the conditions in Definition \ref{type} for the extension $K_iK_j/K_0$. Suppose $ac_{i, m_i}+bc_{j, m_j}=0$ for some $a, b\in k$. Then there is a $\Z/p$-subextension of $K_iK_j/K_0$ defined by $w^p-w=af_i(\frac{1}{t})+bf_j(\frac{1}{t})$, the right-hand-side of which has degree strictly less than $m_i$, i.e. its conductor is strictly less than $m_i$, giving a contradiction. Thus $c_{i,m_1}$ and $c_{j,m_j}$ are linearly independent over $k$.

Finally, suppose $c_{i,m_i}\in\F_p$ for some $i\geq 2$. Then an $\F_p$-linear combination of $w_1$ and $w_i$ generates $\Z/p$-subextension of $K_1K_i/K_0$ having conductor strictly less than $m_i$, again a contradiction. Therefore, $c_{i,m_i}\not\in\F_p$ for $i\geq 2$.
\end{proof}

\begin{remark}\label{1position}
As a variant, the leading coefficient $1$ can instead be put in any one of the $n$ equations.
\end{remark}

\subsection{Key Lemmas}\label{lemmas}
\begin{lemma}\label{higherconductorsp}
Let $G=(\Z/p)^n$, and $k[[z]]/k[[t]]$ be a $G$-cover of type $(m_1+1,\ldots,m_n+1)$, with respect to $G_1,\ldots, G_n$, where $k[[t]]=k[[z]]^G$. Then for $0\leq l\leq n-1$, the $l$-th lower ramification jump of $k((z))/k((t))$ is $$ p^lm_{l+1}-(p-1)\sum_{1\leq i\leq l}p^{i-1}m_i.$$
\end{lemma}

\begin{proof}
Let $m^{(l)}$ denote the $l$-th lower ramification jump of $L/K$. For the base case $l=0$, it follows from the hypothesis and Lemma \ref{ramjumps} that $m^{(0)}=m_1$.

For the induction step, assume that $m^{(j)}=p^j m_{j+1}-(p-1)\sum_{1\leq i\leq j}p^{i-1}m_i$ for all $j\leq l-1$. Let $M/K$ be the $(\Z/p)^{l+1}$-extension $K_0\cdots K_{l+1}/K_0$, and $\Gamma:=\mathrm{Gal}(M/K)$. Recall that $K_i/K_0$ is a $\Z/p$-extension with conductor $m_i+1$. Let $\Gamma_j$ be the $j$-th ramification group of $M/K$ with lower numbering, and let $\varphi_{M/K}(j)$ be the Herbrand function \cite{serre}. Then $\Gamma_j=\Gamma^{\varphi_{M/K}(j)}$, where $\Gamma^i$ is the $i$-th ramification group of $M/K$ with upper numbering. Let $H$ be the subgroup of $\Gamma$ such that $M^H=K_{l+1}$. By Proposition IV.14 in \cite{serre},
\begin{align*}
\Gamma^{i}H/H=(\Gamma/H)^{i}=
\begin{cases}
\Z/p,\ 0\leq i\leq m_{l+1}\\
1,\ i> m_{l+1}.
\end{cases}
\end{align*}
Here, the last equality is due to the isomorphism $\Gamma/H\cong \Z/p$ and the fact that the unique upper jump of $K_{l+1}/K$ equals to its unique lower jump, which is one less than its conductor. Therefore, the $l$-th upper ramification jump of $M/K$ is $m_{l+1}$. Since the upper numbering for ramification groups is compatible with quotients \cite[Proposition 14]{serre}, so are the upper ramification jumps. Thus $m_{l+1}$ is also the $l$-th upper ramification jump of $L/K$. Moreover, $\varphi_{L/K}(m^{(l)})=m_{l+1}$, since $m^{(l)}$ is the $l$-th lower ramification jump of $M/K$.

Now, let $g_j=|I_j|$, where $I_j$ is the $j$-th ramification group of $L/K$ with lower numbering. Observe that $g_j=p^{n-i-1}$ for $m^{(i)}<j\leq m^{(i+1)}$. By the formula on \cite[page 73]{serre} and the induction hypothesis, we have
\begin{align*}
m_{l+1}+1=&1+\varphi_{L/K}(m^{(l)})=\frac{1}{|G|}\sum_{j=0}^{m^{(l)}}g_j\\
=&\frac{1}{p^{n}}\left((m^{(0)}+1)p^{n}+\sum_{i=0}^{l-1}(m^{(i+1)}-m^{(i)})p^{n-i-1}\right)\\
=&m^{(l)}p^{-l}+1+(p-1)p^{-l}\sum_{1\leq j\leq l}p^{j-1}m_j.
\end{align*}
Therefore, the $l$-th ramification jump of $L/K$ is
$m^{(l)}=p^lm_{l+1}-(p-1)\sum_{1\leq i\leq l}p^{i-1}m_i$.
\end{proof}

\begin{remark}
The above proof was based on ideas suggested to the author by Andrew Obus. This lemma can also be proven in a way analogous to Green and Matignon's original proof for the case $\Z/p\times\Z/p$ \cite[Theorem 5.1]{mat}.
\end{remark}

\begin{lemma}\label{genericdifferent}
Let $R[[Z]]/R[[T]]$ be a local $G$-cover, and let $G_1,\ldots, G_n$ be index $p$ subgroups of $G$ such that $R[[Z]]^{G_i}$ and $R[[Z]]^{G_j}$ are linearly disjoint for all $i\neq j$. Suppose $R[[Z]]^{G_1},\ldots,R[[z]]^{G_{n}}$ have branch loci $B_1,\ldots,B_{n}$, each containing $|B_i|=m_i+1$ (geometric) branch points, such that for any $r$ with $1\leq r\leq n$ and any subset $\{B_{i_1},\ldots,B_{i_r}\}$, the cardinality of the set intersection satisfies $\displaystyle|\cap_{1\leq j\leq r}B_{i_j}|=\frac{(\mathrm{min}_j(m_{i_j})+1)(p-1)^{r-1}}{p^{r-1}}$. Then the generic different of $R[[z]]/R[[t]]$ is $\displaystyle\sum_{l=0}^{n-1}(p-1)p^l(m_{l+1}+1)$.
\end{lemma}

\noindent\textit{Proof.}
Since the generic fiber of the lift $R[[Z]]/R[[T]]$ is in characteristic $0$, it is tamely ramified, with $p^{n-1}$ ramification points above each branch point. Thus the generic different is $(p-1)p^{n-1}$ times the total number of branch points, counted without repeat.

Let $B=B_1\cup \cdots \cup B_n$ be the branch locus of $K((Z))/K((T))$. We use the inclusion-exclusion principle to count the number of branch points. For each $1\leq i\leq n$, $\mathrm{min}_j(m_{i_j})+1$ is $m_i+1$ for all $\{B_i, B_{i_2},\ldots,B_{i_k}\}$ such that $i_j\geq i$ for all $j$. There are $\binom{n-i}{k-1}$ such $k$-subsets. Therefore
\begin{align*}
d_\eta&=(p-1)p^{n-1}|B|\\
&=(p-1)p^{n-1}\sum_{k=1}^n(-1)^{k-1}\sum_{i=1}^{n-k+1}\sum_{i_j\geq i\forall j}|B_i\cap B_{i_2}\cap\cdots\cap B_{i_k}|\\
&=(p-1)p^{n-1}\sum_{k=1}^n(-1)^{k-1}\sum_{i=1}^{n-k+1}\binom{n-i}{k-1}(p-1)^{k-1}p^{1-k}(m_i+1)\\
&=\sum_{l=0}^{n-1}(p-1)p^l(m_{l+1}+1).{\hskip 4in} \qed
\end{align*}

\subsection{Main Theorem}\label{mainthm}
We now state our main result, which generalizes Theorem 5.1 of \cite{GM}.

\begin{theorem}[Branch cycle criterion]\label{branchcyclecriterion}
Let $G=(\Z/p)^n$. Suppose $k[[z]]/k[[t]]$ is a $G$-extension of conductor type $(m_1+1,\ldots,m_n+1)$, with respect to $G_1,\ldots, G_n$. Then there is a lift of $G$ to a group of automorphisms of $R[[Z]]$ if and only if the following two conditions hold: 
\begin{enumerate}
\item $m_i\equiv -1$ mod $p^{n-i}$ for $1\leq i\leq n-1$,
\item $k[[z]]^{G_1},\ldots,k[[z]]^{G_{n}}$ can be lifted with branch loci $B_1,\ldots,B_{n}$ such that for any subset of $r$ branch loci $\{B_{i_1},\ldots,B_{i_r}\}$, $\displaystyle|\cap_{1\leq j\leq r}B_{i_j}|=\frac{(\mathrm{min}_j(m_{i_j})+1)(p-1)^{r-1}}{p^{r-1}}$.
\end{enumerate}
\end{theorem}

\begin{proof}
First we show that the combinatorial conditions on the branch loci of lifts of $K_i$ are necessary. Suppose $k[[z]]/k[[t]]$ can be lifted to $R[[Z]]/R[[T]]$. Then so can all the intermediate extensions. We show that for any choice of the set $\{B_{i_1},\ldots,B_{i_r}\}$, $\displaystyle|\cap_{1\leq j\leq r}B_{i_j}|=\frac{\mathrm{min}_j(|B_{i_j}|)(p-1)^{r-1}}{p^{r-1}}$. The base case $r=2$ is shown in \cite[Theorem 5.1]{GM}. Suppose this is true for $r\leq l$, and consider the extension $K_0\cdots K_{l+1}/K_0$. Then the number of branch points in the lift of $K_0\cdots K_{l+1}/K_0\cdots K_l$ is $(p-1)p^l$ times the number of branch points in the lift of $K_{l+1}/K_0$ that are not in that of any $K_i/K_0$ for $1\leq i\leq l$. Write $d=|B_1\cap\ldots\cap B_{l+1}|$. Using Lemma \ref{genericdifferent}, the degree of the generic different of $K_0\cdots K_{l+1}/K_0\cdots K_l$ is given by
\begin{align*}
d_\eta&=(p-1)p^l\left(|B_{l+1}|+\sum_{r=1}^l(-1)^r\sum_{i=1}^{l-r+1}\sum_{i_j\geq i,\forall 2\leq j\leq r}|B_i\cap B_{i_2}\cap\ldots\cap B_{i_r}\cap B_{l+1}|\right)\\
&=(p-1)p^l\left(|B_{l+1}|-(p-1)p^{-1}\sum_{r=1}^{l-1}(-1)^{r-1}\sum_{i=1}^{l-r+1}\sum_{i_j\geq i,\forall j}|B_i\cap B_{i_2}\cap\ldots\cap B_{i_r}|+(-1)^ld\right)\\
&=(p-1)\biggl(p^l(m_{l+1}+1)-\Bigl(d_{\eta,K_0\cdots K_l/K_0}-(p-1)(-p)^{l-1}(m_1+1)(p-1)^{l-1}p^{1-l}\Bigr)+p^l(-1)^ld\biggr)\\
&=(p-1)\left(p^l(m_{l+1}+1)-\sum_{i=1}^l(p-1)p^{i-1}(m_i+1)+(-1)^{l-1}(p-1)^l(m_1+1)+p^l(-1)^ld\right).
\end{align*}

By the different criterion \cite[Section 3.4]{GM}, this equals the degree of special different, which in this case equals $p-1$ times the conductor of $K_0\cdots K_{l+1}/K_0\cdots K_l$. Recall from Lemma \ref{ramjumps} this is $$d_{s,K_0\cdots K_{l+1}/K_0\cdots K_l}=(p-1)\Bigl(p^lm_{l+1}-(p-1)\sum_{i=1}^lp^{i-1}m_i+1\Bigr).$$
Thus $(-1)^{l-1}(p-1)^l(m_1+1)+p^l(-1)^ld=0$.

Hence the $l+1$ lifts share $\displaystyle|B_1\cap\ldots\cap B_{l+1}|=d=\frac{(m_1+1)(p-1)^l}{p^l}=\frac{\mathrm{min}_i(|B_i|)(p-1)^{l+1-1}}{p^{l+1-1}}$ common branch points, proving that the conditions on the sets $B_i$ are necessary. Furthermore, since the number of common branch points is an integer, this also shows that the congruence conditions $m_i\equiv -1$ mod $p^{n-i}$ for $1\leq i\leq n-1$ are necessary.

Finally, we show that the conditions on the branch loci are sufficient.

We have from the beginning of the proof that the $l$-th lower ramification jump is $$ p^lm_{l+1}-(p-1)\sum_{1\leq i\leq l}p^{i-1}m_i.$$ Therefore, by Lemma \ref{ramjumps} the degree of the different of $k[[z]]/k[[t]]$ is
\begin{align*}
d_s&=\sum_{l=0}^{n-1}(m_1^{(l)}+1)(p-1)p^{n-l-1}\\
&=\sum_{l=0}^{n-1}(p-1)p^{n-l-1}\left(p^lm_{l+1}-(p-1)\sum_{i=1}^lp^{i-1}m_i+1\right)\\
&=(p-1)\sum_{l=0}^{n-1}p^l(m_{l+1}+1).
\end{align*}

The degree of the generic different of $k[[z]]/k[[t]]$ is $$ d_\eta=\sum_{l=0}^{n-1}(p-1)p^l(m_{l+1}+1)=d_s.$$

It thus follows from the different criterion \cite[Section 3.4]{GM} that $G$ lifts to a group of automorphisms of $R[[Z]]$.
\end{proof}

Theorem \ref{branchcyclecriterion} can also be stated in terms of the number of branch points with a certain monodromy subgroup, which does not require choosing the $n$ generating $\Z/p$-subcovers and counting common branch points. This leads to the next example.

\begin{example}
For a $(\Z/2)^3$-cover of type $(4,4,4)$, the criterion asserts that the non-identity elements of the group $(\Z/2)^3$ are in bijection with the branch points of any given lift, each generating the inertia group at exactly one point.
\end{example}

\begin{example}
More generally, for each element of $(\Z/2)^3$, we can write down how many branch points of the lift have stabilizers generated by that element.

Let $G=(\Z/2)^3$ with elements $\{1,a,b,c,ab,ac,bc,abc\}$, and consider $(\Z/2)^2$-subgroups of $G$, $G_1=\langle a,b\rangle,G_2=\langle a,c\rangle,G_3=\langle b,c\rangle$. Suppose $k[[z]]/k[[t]]$ is a $G$-extension of conductor type $(m_1+1,m_2+1,m_3+1)$ with respect to $G_1,G_2,G_3$. Let $C_1,C_2,C_3$ be lifts of the three generating $\Z/2$-subcovers $k[[z]]^{G_1}/k[[t]], k[[z]]^{G_2}/k[[t]],k[[z]]^{G_3}/k[[t]]$ respectively. Then for each $C_i$ and element $g\in G$, we can write down the numbers $m(g)$ of branch points of $C_i$ which have inertia group generated by $g$, as in the table below.

Here, $abc$ is the only element that is not in $G_i$ for any $i=1,2,3$, so the number of branch points shared by all $C_1,C_2,C_3$ is $m(abc)=\frac{m_1+1}{4}$. The element $bc$ is not in $G_1$ or $G_2$, so the number of branch points shared by $C_1$ and $C_2$ is $m(abc)+m(bc)=\frac{m_1+1}{2}$. Similarly, $ac$ is not in $G_1$ or $G_3$, so the number of branch points shared by $C_1$ and $C_3$ is $m(abc)+m(ac)=\frac{m_1+1}{2}$. Also, $ab$ is not in $G_2$ or $G_3$, so the number of branch points shared by $C_2$ and $C_3$ is $m(abc)+m(ab)=\frac{m_2+1}{2}$. This is exactly the combinatorial condition on the branch points of the lift in Theorem \ref{branchcyclecriterion}.
\end{example}
\begin{center}
\begin{tabular}{|c|c|c|c|}
    \hline
    \text{Elements} & $C_1$ & $C_2$ & $C_3$ \\
    \hline
    a & & & $m_3+1-\frac{m_2+1}{2}-\frac{m_1+1}{4}$\\
    \hline
    b & & $\frac{m_2+1}{2}-\frac{m_1+1}{4}$ & \\
    \hline
    c & $\frac{m_1+1}{4}$ & & \\
    \hline
    ab & & $\frac{m_2+1}{2}-\frac{m_1+1}{4}$ & $\frac{m_2+1}{2}-\frac{m_1+1}{4}$ \\
    \hline
    ac & $\frac{m_1+1}{4}$ & & $\frac{m_1+1}{4}$ \\
    \hline
    bc & $\frac{m_1+1}{4}$ & $\frac{m_1+1}{4}$ & \\
    \hline
    abc & $\frac{m_1+1}{4}$ & $\frac{m_1+1}{4}$ & $\frac{m_1+1}{4}$\\
    \hline
    Total & $m_1+1$ & $m_2+1$ & $m_3+1$\\
    \hline
\end{tabular}
\end{center}

\section{Coalescing of Branch Points}

In this section, we first recall a result in Pries-Zhu \cite{pries-zhu}, on stratification of the space of Artin-Schreier covers. We then give an interpretation of the result in terms of branch loci of lifts of these covers. We give a description of the coalescing behavior of the branch points of the lifts, which will be used in Section 5.

\subsection{Stratification of the Space of Artin-Schreier Covers}\label{pz}
Consider a smooth projective curve $X$ over $k$ of genus $g$. The \textit{$p$-rank} of $X$ is the integer $s$ such that the cardinality of $\mathrm{Jac}(X)[p](k)$ is $p^s$. We have that $0 \leq s \leq g$. For $g=1$, the $p$-rank is also called the \textit{Hasse invariant}.

Now let $X\to \P^1_k$ be an Artin-Schreier cover in characteristic $p$. Then $s = r(p-1)$ for some integer $r \geq 0$ \cite{pries-zhu}. We can study the stratification of $AS_g$, the moduli space of Artin-Schreier covers of genus $g$, by $p$-rank, into strata $AS_{g,s}$ consisting of covers with $p$-rank $s$. By the Riemann-Hurwitz formula, $2g-2 = p(-2)+\mathrm{deg}(D)$, where $D$ is the divisor carved out by the different, called the ramification divisor. As we will show below, $g=d(p-1)/2$ for some integer $d$. Assume $g\geq 1$. Then we have the following result:
    
\begin{theorem}[Pries-Zhu, 2010]\label{pries-zhu}
\begin{enumerate}
\item The set of irreducible components of $AS_{g,s}$ is in natural bijection with the set of partitions $[e_1,...,e_{r+1}]$ of $d+2$ into $r+1$ positive integers such that each $e_j \not\equiv 1$ mod $p$.
\item The irreducible component of $AS_{g,s}$ for the partition $[e_1,...,e_{r+1}]$ has dimension $d-1-\sum_{j=1}^{r+1}\lfloor(e_j-1)/p\rfloor$.
\end{enumerate}
\end{theorem}

In fact, the bijection in part 1 can be given explicitly. Since $k$ is algebraically closed, after some automorphism of $\P^1_k$, we can assume that $f$ is not branched at $\infty$. Thus $f: X\to \P^1_k$ is given by an equation $\displaystyle y^p-y=\sum_{i=1}^{n}f_i(\frac{1}{x-c_i})$, where $f_i$ are polynomials over $k$ of degrees not divisible by $p$ (in particular $\mathrm{deg}(f_i)>0$), and $c_i\in k$ are distinct. This cover is branched at $n$ points, $\{c_1,\ldots,c_n\}$, and any Artin-Schreier cover branched at these points is of the above form. Let $s$ be the $p$-rank of $f$. The Deuring-Shafarevich theorem \cite[Theorem 4.1]{subrao} states that, for $X\xrightarrow{\Z/p}Y$ over $k$, $s_X-1=p(s_Y-1)+n(p-1)$, where $n$ is the number of branch points on $Y$. Here the $p$-rank of $Y=\P^1_k$ is $0$. Therefore $s=(n-1)(p-1)$, and $n=r+1$, with $r$ defined as above.

Let $e_i=\mathrm{deg}(f_i)+1$. Then $e_i\not\equiv 1$ mod $p$. By the Riemann-Hurwitz formula, \cite[IV.2, Proposition 4]{serre} and Remark \cite[IV.2, Exercise 5b]{serre},
\begin{align*}
2g-2&=p(0-2)+\mathrm{deg}(D)=-2p+\sum_{i=1}^{r+1}\sum_{j=0}^\infty (|G^i_j|-1)=-2p+\sum_{i=1}^{r+1}e_i(p-1),
\end{align*}
where $G^i_j$ is the $j$-th lower ramification group for the local extension at the branch point $c_i$.

Thus $\displaystyle g=(p-1)(\sum_{i=1}^{r+1}e_i-2)/2=d(p-1)/2$ for an integer $d$, so $\displaystyle\sum_{i=1}^{r+1}e_i=d+2$ and $[e_1,\ldots,e_{r+1}]$ is a partition of $d+2$.

\subsection{Coalescing of Branch Points of a Lift}\label{coalesce}
Recall that by the Oort conjecture, every $\Z/p$-cover lifts. In this subsection we study how the branch points of the lift coalesce on the special fiber.

\begin{theorem}\label{branchcoalesce}
With the above notation, consider the component of $AS_{g,s}$ of an Artin-Schreier cover $f: X\to \P^1_k$ with $p$-rank $s$ which corresponds to the partition $[e_1,...,e_{r+1}]$ of $d+2$, with each $e_j \not\equiv 1$ mod $p$. Suppose $f$ is branched at $\{c_1,\ldots, c_{r+1}\}$, given by an equation of the form $\displaystyle y^p-y=\sum_{i=1}^{r+1}f_i(\frac{1}{x-c_i})$, where $e_i=\mathrm{deg}(f_i)+1$. Then there exists a lift of $f$ to $R$ whose generic fiber is a degree $p$ Kummer cover with $d+2$ branch points, $e_i$ of which coalesce to $c_i$ on $\P^1_k$ for $1\leq i\leq r+1$.

Conversely, any lift of $f$ is a $\Z/p$-cover with $d+2$ branch points, $e_i$ of which coalesce to $c_i$ on $\P^1_k$ for $1\leq i\leq r+1$.
\end{theorem}

\begin{proof}
Localizing at each branch point of $f$, we get $r+1$ local extensions, of $k[[x-c_i]], 1\leq i \leq r+1$. Since $x-c_j$ is a unit in $k[[x-c_i]]$ for all $j\neq i$, $\frac{1}{x-c_j}\in k[[x-c_i]]$ and thus $f_j(\frac{1}{x-c_j})\in k[[x-c_i]]$. Then there exists an element $z=-(f_j+f_j^p+f_j^{p^2}+\cdots)$ in $k[[x-c_i]]$ such that $z^p-z=f_j$. Therefore, after a change of variables, the local extension of $k[[x-c_i]]$ is given generically by $\displaystyle y^p-y=f_i(\frac{1}{x-c_i})$.

By the Oort conjecture, after possibly extending $R$, we can lift these local covers, which gives us branched covers of $\mathrm{Spec}\,R[[x-c_i]]$, branched at $b_i$ points on the generic fiber, for some $b_i>0$, all coalescing at $c_i$. By the different criterion \cite[Section 3.4]{GM}, the generic different, $b_i(p-1)$, is equal to the special different, $(\mathrm{deg}(f_i)+1)(p-1)$, so $b_i=\mathrm{deg}(f_i)+1=e_i$. By the proof of Theorem 2.2 in \cite{CGH}, we can patch these local lifts together to get a smooth $\Z/p$-cover $X_R\to \P^1_R$, with $e_i$ branch points coalescing to the point $c_i$ on $\P^1_k$.

Let $X_K\to \P^1_K$, branched at $m$ points, be the generic fiber of the lift $X_R\to \P^1_R$. Then by the Riemann-Hurwitz formula and flatness of $X_R\to \P^1_R$, $(m-2)(p-1)/2=g_{X_K}=g_{X}=d(p-1)/2$, so $\displaystyle m=\sum_{i=1}^{r+1}e_i=d+2$.

Now we prove the converse. Suppose $F: X_R\to\P^1_R$ is a lift of $f$. Localizing $\P^1_R$ at the closed point $c_i\in\P^1_k$, for $1\leq i\leq r+1$, we get the inclusion $\mathrm{Spec}\,\hat{\mathcal{O}}_{\P^1_R,c_i}\to \P^1_R$. Now taking its fiber product with $F$, we get an extension $R[[z]]$ of $R[[x-c_i]]$ branched at only those branch points of $F$ coalescing at $c_i$. Suppose there are $n_i$ of them.

The reduction of $R[[z]]/R[[x-c_i]]$ is an extension of $k[[x-c_i]]$ given generically by $\displaystyle y^p-y=f_i(\frac{1}{x-c_i})$, as shown above. Again, by the different criterion, $R[[z]]/R[[x-c_i]]$ has to be branched at $\mathrm{deg}(f_i)+1=e_i$ points. Therefore, $n_i=e_i\not\equiv 1$ mod $p$.
\end{proof}

\begin{remark}
We can therefore interpret Theorem \ref{pries-zhu} as a description of $K$-covers $f:X\to \P^1_K$ with good reduction, in terms of how their branch points coalesce on the special fiber. Namely, if $X_R\to\P^1_R$ is the smooth model of $f$, then $e_i$ points on $\P^1_R$ coalesce to the $i$-th branch point on the special fiber. Moreover, let $\mathcal{H}_{m,p}$ be the space of $p$-covers of $\P^1_K$ branched at $m$ points, and let $\mathcal{H}^{good}_{m,p}$ be the subspace of $\mathcal{H}_{m,p}$ consisting of those covers having good reduction. Then we get a stratification of $\mathcal{H}^{good}_{m,p}$ into strata $\mathcal{H}^{good}_{m,p,n}$ of covers whose reduction have $n$ branch points. This can also be used to find criteria for potentially good reduction for covers in charateristic $0$ with general branch locus geometry, which would generalize the criterion in \cite{lehr}.
\end{remark}

\begin{remark}
Part 2 of Theorem \ref{pries-zhu} can be used to describe the strata $\mathcal{H}^{good}_{m,p,n}$ in characteristic $0$. Since we construct lifts to $R$ by lifting the coefficients of the defining polynomials for the covers over $k$, the component of $\mathcal{H}^{good}_{m,p,n}$ consisting of covers with branch locus partition $[e_1,\ldots, e_n]$, where $e_i$ points coalesce to one point for each $i$, is a $p$-adic neighborhood of a subvariety of $\mathcal{H}^{good}_{m,p}$ that has dimension equal to $m-3-\sum_{i=1}^n\lfloor(e_i-1)/p\rfloor$.
\end{remark}

\begin{corollary}\label{even coalesce}
Let $f:X\to \P^1_R$ be a lift of a $\Z/2$-cover of $\P^1_k$. Then the number of branch points of $f$ coalescing to one point over $k$ is even.
\end{corollary}

\begin{proof}
By the above theorem, the number of branch points of $f$ coalescing to the $i$-th branch point on the special fiber is $e_i\not\equiv 1$ mod $2$, i.e. $e_i$ is even.
\end{proof}

We can apply this theorem to $\Z/2$-covers in characteristic $2$, and look at how branch points of lifts of elliptic covers coalesce on the special fiber.

\begin{example}
Let $p=2$, and $g=1$. Consider $X\to \P^1_k$, where $X$ is an elliptic curve. Then $d=4$ in the notation of Section \ref{pz}. The space of elliptic curves is parameterized by the $j$-line, where elliptic curves of $j$-invariant $0$ have $2$-rank $s=0$, and elliptic curves of $j$-invariant non-zero have $2$-rank $s=1$.

Case 1: $s=0$, $r=s/(p-1)=0$. Then $X\to P^1_k$ is branched at $r+1=1$ point, corresponding to the partition $(4)$ of $4$ into $1$ even integer. Any lift $X_R\to \P^1_R$ has $4$ branch points, all of which coalesce to one point on the special fiber. $X_R$ has $j$-invariant $j\in \mathfrak{m}$, and is in the subspace $\mathcal{H}^{good}_{4,2,1}$ of $\mathcal{H}^{good}_{4,2}\cong \mathbb{A}^1_R$ defined by $v(j)>0$.

Case 2: $s=1$, $r=s/(p-1)=1$. Then $X\to P^1_k$ is branched at $r+1=2$ points, corresponding to the partition $(2,2)$ of $4$ into $2$ even integers. Any lift $X_R\to \P^1_R$ has two pairs of $2$ branch points, each pair of which coalesce to one point on the special fiber. $X_R$ has $j$-invariant $j\in R^*$, and is in the subspace $\mathcal{H}^{good}_{4,2,2}$ of $\mathcal{H}^{good}_{4,2}\cong \mathbb{A}^1_R$ defined by $v(j)=0$.
\end{example}

\section{Lifts of $(\Z/2)^3$-Covers}

In this section, we apply results in the previous two sections to construct explicit lifts for $(\Z/2)^3$-covers of various conductor types. We first use Mitchell's classification \cite{mitchell} to show that covers of type $(4,4,4)$ can be lifted only with equidistant geometry. (Here, we say that a branch locus $B$ has equidistant geometry if $d_p(b_i - b_j ) = \rho$, with $\rho \geq 0$
fixed, for all pairs of distinct branch points $b_i, b_j \in B$.) Then we construct lifts for all covers of non-constant conductor type $(4,4,2r)$, $r\geq 3$, with certain branch locus geometry, and show that the lifts can never be equidistant.

\subsection{Hurwitz Trees for $(\Z/2)^3$-Covers of Type $(4,4,4)$}\label{444}

\begin{definition}
Let $T$ be a rooted tree, where the root node $v_0$ is connected to one other node $v_1$. Define a partial ordering on vertices $w$ of $T$, by inclusion of paths from $v_0$ to $w$. For each vertex $w_i$ of $T$ connected to $v_1$ that is not $v_0$, consider the subtree consisting of vertices greater than or equal to $w_i$, with $w_i$ as the root of the subtree. We call this subtree of $T$ \textit{the $i$-th branch of $T$}. A branch has \textit{size} $b_i$ if it contains $b_i$ leaf (terminal) nodes.
\end{definition}

For a cover $R[[Z]]/R[[T]]$, we can build a rooted tree, called the \textit{Hurwitz tree}, from the dual graph $\Gamma$ of the semi-stable reduction $X_k$. For a precise description of a Hurwitz tree, as a rooted metric tree, with associated characters and differential data, see \cite{BW}.

Vertices and edges of $\Gamma$ correspond to irreducible components and nodes of $X_k$, and there is an edge between two vertices if and only if their corresponding irreducible components intersect. Next, we append a vertex $v_0$, connected via an edge $e_0$, to the vertex $v_1$ corresponding to the component $\infty$ specializes to. Call this the \textit{root node} of the Hurwitz tree. Finally, for each $b_i\in B$, append a vertex $x_i$, via an edge $e_i$, to the vertex $w_j$, corresponding to the component $b_i$ specializes to. Call these the \textit{leaf nodes} of the Hurwitz tree.

\begin{remark}
Branch points corresponding to leaves in the same branch are $p$-adically closer to each other than they are to branch points corresponding to leaves in other branches.
\end{remark}

In order to simplify the notations, we will only consider the sizes of the branches of a Hurwitz tree, ignoring further structure of the tree. In each particular case, we will specify whether the leaves in a branch are equidistant, or there is further branching.

\begin{definition}
We say that a Hurwitz tree has \textit{branch partition} $(e_1,\ldots,e_k)$, if there are $k$ branches, with the $i$-th branch having size $e_i$. We say that a characteristic $0$ cover has \textit{branch locus geometry} $(e_1,\ldots, e_k)$ if its Hurwitz tree has branch partition $(e_1,\ldots,e_k)$, whose leaves correspond to branch points of the cover.
\end{definition}

We first introduce the classification of Hurwitz trees for $(\Z/2)^2$-covers of type $(4,4)$ by Mitchell.

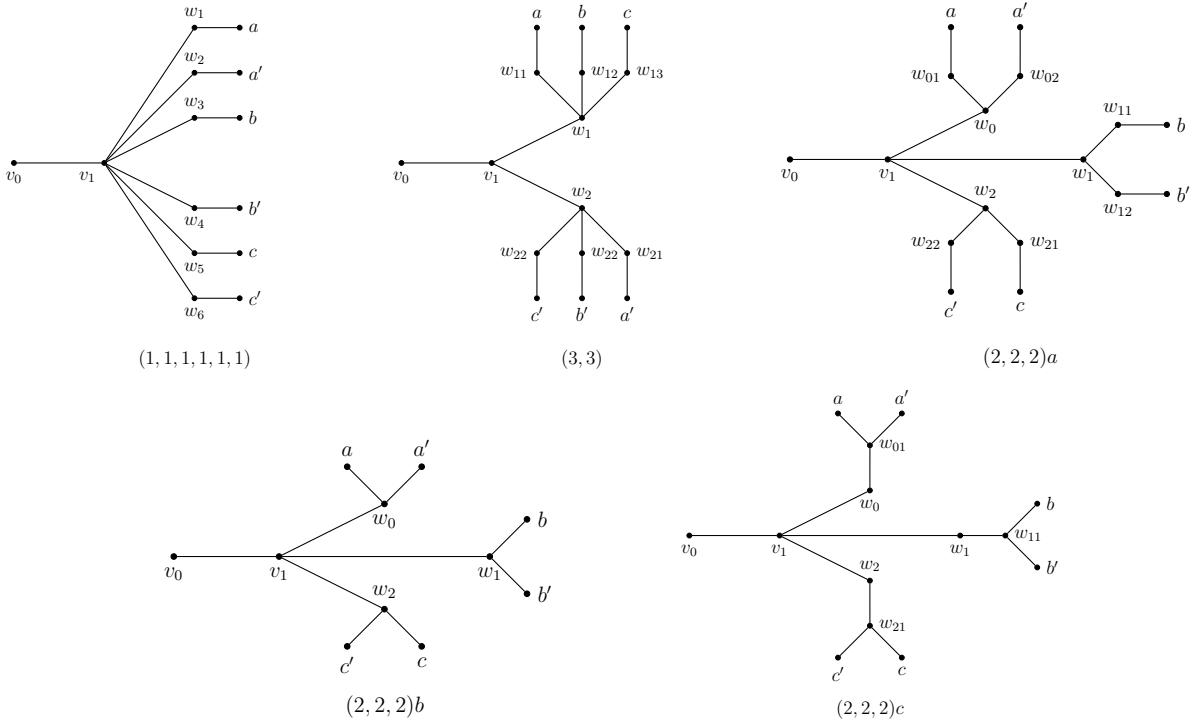
\begin{figure}[h]
\begin{center}

\begin{tikzpicture}[node distance = 1cm, scale=0.8, every node/.style={scale=0.6}]
	\tikzstyle{main node}=[draw,circle,fill=black,minimum size=3pt,
	inner sep=0pt]
	\tikzstyle{empty node} = [draw, color=white, fill=white,opacity = 0]
	
	\node[main node] (v0) [label=below:$v_0$] {};
	\node[empty node] (emp) [right of =v0] {};
	\node[main node] (v1) [right of =emp,label=below left :$v_1$] {};
	\node[empty node] (emp1) [right of = v1] {};
	\node[empty node] (emp2) [right of = emp1] {};
	\node[main node] (w3) [above of=emp2, label=above: $w_3$] {};
	\node[main node] (w2) [above of=w3,label=above: $w_2$] {};
	\node[main node] (w1) [above of=w2,label=above: $w_1$] {};
	\node[main node] (w4) [below of=emp2,label=below: $w_4$] {};
	\node[main node] (w5) [below of=w4,label=below: $w_5$] {};
	\node[main node] (w6) [below of=w5, label=below: $w_6$] {};
	\node[main node] (a) [right of= w1, label=right: $a$] {};
	\node[main node] (a') [right of= w2, label=right: $a'$] {};
	\node[main node] (b) [right of= w3, label=right: $b$] {};
	\node[main node] (b') [right of= w4, label=right: $b'$] {};
	\node[main node] (c) [right of= w5, label=right: $c$] {};
	\node[main node] (c') [right of= w6, label=right: $c'$] {};
	
	\draw (v0) -- (v1); 
	\draw (v1) -- (w1);
	\draw (v1) -- (w2);
	\draw (v1) -- (w3);
	\draw (v1) -- (w4);
	\draw (v1) -- (w5);
	\draw (v1) -- (w6);
	\draw (w1) -- (a);
	\draw (w2) -- (a');
	\draw (w3) -- (b);
	\draw (w4) -- (b');
	\draw (w5) -- (c);
	\draw (w6) -- (c');
	
	\node [below=1cm, align=flush center,text width=8cm] at (w6)
        {
            $(1,1,1,1,1,1)$
        };
	
\end{tikzpicture}
\hspace{-0.2cm}
\begin{tikzpicture}[node distance = 1cm, scale=0.8, every node/.style={scale=0.6}]
	\tikzstyle{main node}=[draw,circle,fill=black,minimum size=3pt,
	inner sep=0pt]
	\tikzstyle{empty node} = [draw, color=white, fill=white,opacity = 0]
	
	\node[main node] (v0) [label=below:$v_0$] {};
	\node[empty node] (emp) [right of=v0] {};
	\node[main node] (v1) [right of =emp,label=below:$v_1$] {};
	\node[empty node] (emp1) [right of=v1] {};
	\node[empty node] (emp2) [right of=emp1] {};
	\node[empty node] (emp3) [right of=emp2] {};
	\node[main node] (w1) [above of =emp2,label=below:$w_1$] {};
	\node[main node] (w12) [above of = w1, label= right: $w_{12}$] {};
	\node[main node] (w11) [left of =w12,label=left:$w_{11}$] {};
	\node[main node] (w13) [right of =w12,label=right:$w_{13}$] {};
	\node[main node] (w2) [below of =emp2,label=above:$w_2$] {};
	\node[main node] (w22) [below of = w2, label= right: $w_{22}$] {};
	\node[main node] (w21) [right of =w22,label=right:$w_{21}$] {};
	\node[main node] (w23) [left of =w22, label=left:$w_{22}$] {};
	\node[main node] (a) [above of = w11, label=above: $a$] {};
	\node[main node] (b) [above of =w12, label=above:$b$] {};
	\node[main node] (c) [above of = w13, label=above: $c$] {};
	\node[main node] (a') [below of = w21, label=below: $a'$] {};
	\node[main node] (b') [below of =w22, label=below:$b'$] {};
	\node[main node] (c') [below of = w23, label=below: $c'$] {};
	
	\draw (v0) -- (v1);
	\draw (v1) -- (w1);
	\draw (v1) -- (w2);
	\draw (w1) -- (w11);
	\draw (w1) -- (w12);
	\draw (w1) -- (w13) ;
	
	\draw (w2) -- (w21);
	\draw (w2) -- (w22);
	\draw (w2) -- (w23);
	\draw (w11) -- (a);
	\draw (w12) -- (b);
	\draw (w13) -- (c);
	\draw (w21) -- (a');
	\draw (w22) -- (b'); 
	\draw (w23) -- (c'); 
	
	\node [below=2cm, align=flush center,text width=8cm] at (w22)
        {
            $(3,3)$
        };
\end{tikzpicture}
\hspace{-0.2cm}
\begin{tikzpicture}[node distance = 1cm, scale=0.8, every node/.style={scale=0.65}]
	\tikzstyle{main node}=[draw,circle,fill=black,minimum size=3pt,
	inner sep=0pt]
	\tikzstyle{empty node} = [draw, color=white, fill=white,opacity = 0]
	
	\node[main node] (v0) [label=below:$v_0$] {};
	\node[empty node] (emp) [right of=v0] {};
	\node[main node] (v1) [right of =emp,label=below:$v_1$] {};
	\node[empty node] (emp1) [right of=v1] {};
	\node[empty node] (emp2) [right of=emp1] {};
	\node[empty node] (emp3) [right of=emp2] {};
	\node[main node] (w0) [above of =emp2,label=below:$w_0$] {};
	\node[main node] (w01) [above left of =w0,label=left:$w_{01}$] {};
	\node[main node] (w02) [above right of =w0,label=right:$w_{02}$] {};
	\node[main node] (w1) [right of =emp3,label=below:$w_1$] {};
	\node[main node] (w11) [above right of =w1,label=above:$w_{11}$] {};
	\node[main node] (w12) [below right of =w1,label=below:$w_{12}$] {};
	\node[main node] (w2) [below of =emp2,label=above:$w_2$] {};
	\node[main node] (w21) [below right of =w2,label=right:$w_{21}$] {};
	\node[main node] (w22) [below left of =w2, label=left:$w_{22}$] {};
	\node[main node] (a) [above of = w01, label=above: $a$] {};
	\node[main node] (a') [above of = w02, label=above: $a'$] {};
	\node[main node] (b) [right of = w11, label=right: $b$] {};
	\node[main node] (b') [right of = w12, label=right: $b'$] {};
	\node[main node] (c) [below of = w21, label=below: $c$] {};
	\node[main node] (c') [below of = w22, label=below: $c'$] {};
	
	\draw (v0) -- (v1);
	\draw (v1) -- (w0);
	\draw (v1) -- (w1);
	\draw (v1) -- (w2);
	\draw (w0) -- (w01);
	\draw (w0) -- (w02); 
	\draw (w1) -- (w11);
	\draw (w1) -- (w12);
	\draw (w2) -- (w21);
	\draw (w2) -- (w22);
	\draw (w01) -- (a);
	\draw (w02) -- (a');
	\draw (w11) -- (b);
	\draw (w12) -- (b');
	\draw (w21) -- (c);
	\draw (w22) -- (c');
	
	\node [below=2cm, align=flush center,text width=8cm] at (w21)
        {
            $(2,2,2)a$
        };
\end{tikzpicture}
\end{center}

\begin{center}
\begin{tikzpicture}[node distance = 1cm, scale=0.8, every node/.style={scale=0.7}]
	\tikzstyle{main node}=[draw,circle,fill=black,minimum size=3pt,
	inner sep=0pt]
	\tikzstyle{empty node} = [draw, color=white, fill=white,opacity = 0]
	
	\node[main node] (v0) [label=below:$v_0$] {};
	\node[empty node] (emp) [right of=v0] {};
	\node[main node] (v1) [right of =emp,label=below:$v_1$] {};
	\node[empty node] (emp1) [right of=v1] {};
	\node[empty node] (emp2) [right of=emp1] {};
	\node[empty node] (emp3) [right of=emp2] {};
	\node[main node] (w0) [above of =emp2,label=below:$w_0$] {};
	
	\node[main node] (w1) [right of =emp3,label=below:$w_1$] {};
	\node[main node] (w2) [below of =emp2,label=above:$w_2$] {};
	\node[main node] (a) [above left of = w0, label=above: $a$] {};
	\node[main node] (a') [above right of = w0, label=above: $a'$] {};
	\node[main node] (b) [above right of = w1, label=right: $b$] {};
	\node[main node] (b') [below right of = w1, label=right: $b'$] {};
	\node[main node] (c) [below right of = w2, label=below: $c$] {};
	\node[main node] (c') [below left of = w2, label=below: $c'$] {};
	
	\draw (v0) -- (v1);
	\draw (v1) -- (w0);
	\draw (v1) -- (w1);
	\draw (v1) -- (w2);
	\draw (w0) -- (a);
	\draw (w0) -- (a');
	\draw (w1) -- (b);
	\draw (w1) -- (b');
	\draw (w2) -- (c);
	\draw (w2) -- (c');
	
	\node [below=1.5cm, align=flush center,text width=8cm] at (w2)
        {
            $(2,2,2)b$
        };
\end{tikzpicture}
\qquad
\begin{tikzpicture}[node distance = 1cm, scale=0.8, every node/.style={scale=0.6}]
	\tikzstyle{main node}=[draw,circle,fill=black,minimum size=3pt,
	inner sep=0pt]
	\tikzstyle{empty node} = [draw, color = white, fill = white, opacity = 0]
	
	\node[main node] (v0) [label=below:$v_0$] {};
	\node[empty node] (emp) [right of=v0] {};
	\node[main node] (v1) [right of =emp,label=below:$v_1$] {};
	\node[empty node] (emp1) [right of=v1] {};
	\node[empty node] (emp2) [right of=emp1] {};
	\node[empty node] (emp3) [right of=emp2] {};
	\node[main node] (w0) [above of =emp2,label=below:$w_0$] {};
	\node[main node] (w01) [above of =w0,label=right:$w_{01}$] {};
	
	\node[main node] (w1) [right of =emp3,label=below:$w_1$] {};
	\node[main node] (w11) [right of =w1,label=right:$w_{11}$] {};
	\node[main node] (w2) [below of =emp2,label=above:$w_2$] {};
	\node[main node] (w21) [below of =w2,label=right:$w_{21}$] {};
	\node[main node] (a) [above left of = w01, label=above: $a$] {};
	\node[main node] (a') [above right of = w01, label=above: $a'$] {};
	\node[main node] (b) [above right of = w11, label=right: $b$] {};
	\node[main node] (b') [below right of = w11, label=right: $b'$] {};
	\node[main node] (c) [below right of = w21, label=below: $c$] {};
	\node[main node] (c') [below left of = w21, label=below: $c'$] {};
	
	\draw (v0) -- (v1);
	\draw (v1) -- (w0);
	\draw (v1) -- (w1);
	\draw (v1) -- (w2);
	\draw (w0) -- (w01);
	\draw (w1) -- (w11); 
	\draw (w2) -- (w21);
	\draw (w01) -- (a);
	\draw (w01) -- (a');
	\draw (w11) -- (b);
	\draw (w11) -- (b');
	\draw (w21) -- (c);
	\draw (w21) -- (c');
	
	\node [below=1.5cm, align=flush center,text width=8cm] at (w21)
        {
            $(2,2,2)c$
        };
\end{tikzpicture}

\end{center}
\caption{Hurwitz trees for Klein-four covers of type (4,4)}
\label{44trees}
\end{figure}

\begin{theorem}[\cite{mitchell}, Theorem 3.4.14]\label{44}
The only possible Hurwitz trees for a $(\Z/2)^2$-cover over $R$ of type $(4,4)$ are the ones in Figure \ref{44trees}. In particular, the only possible branch partitions are $(1,1,1,1,1,1), (3,3)$ and $(2,2,2)$.
\end{theorem}

The proof of the theorem involves looking at each possible tree with the correct total number of leaves, and checking if each path in the tree satisfies the depth relation \cite[Definition 3.1.4, H4]{mitchell}.

\begin{lemma}\label{211}
A lift of a $\Z/2$-cover of conductor $4$ cannot have branch partition $(2,1,1)$.
\end{lemma}

\begin{proof}
After a change of variables, we may suppose that the extension $k((y))/k((t))$ is defined by $y^2-y=\frac{1}{t^3}+\frac{\alpha}{t}$, and that it has a lift defined by $Y^2-Y=\frac{1}{T^3}+\frac{A}{T}$, where $A\in R$ reduces to $\alpha$. After the further change of variable $Z=T^2Y$, $R[[Y]]/R[[T]]$ is given by $Z^2-T^2Z=T+AT^3$, and is thus ramified at the roots of the discriminant $T^4+4T+4AT^3=T(T^3+4AT^2+4)$. Let $0,a,b,c\in R$ be the branch points of $R[[Y]]/R[[T]]$, where $(T-a)(T-b)(T-c)=T^3+4AT^2+4$ (after enlarging $R$). Suppose that $R[[Y]]/R[[T]]$ has branch partition $(2,1,1)$. Without loss of generality, we may assume that $v(a)>v(b)\geq v(c)\geq 0$, where $v(2)=1$ is the normalized valuation on $R$. Here $a$ and $0$ are on the same branch of the Hurwitz tree, and $b$ and $c$ are on separate branches.

Then we have that $abc=-4, ab+ac+bc=0, a+b+c=-4A$. The first equality gives that $v(a)+v(b)+v(c)=2$, so $2/3<v(a)\leq 2$. Now we consider two cases:

Case 1: $v(a)\neq v(b+c)$. Then from the third equality above, $v(a+b+c)=\mathrm{min}(v(a),v(b+c))=v(-4A)\geq 2$. Since $v(a)\leq 2$, we must have that $v(b+c)>v(a)=2$. Since the three branch points $0,b,c$ are equidistant, we get
\begin{align*}
v(b-0)=v(c-0)=v(b-c).
\end{align*}
Then $v(b-c)<v(c)+1=v(2c)$, and so
\begin{align*}
v(a)<v(b+c)=v(b-c+2c)=\mathrm{min}(v(b-c),v(c)+1)=v(b-c)=v(b).
\end{align*}
This contradicts our assumption that $v(a)>v(b)$.

Case 2: $v(a)=v(b+c)$. From $ab+ac+bc=0$, we get $v(a(b+c))=v(bc)$. However, $v(a(b+c))=v(a)+v(b+c)=2v(a)$, and $v(bc)=v(b)+v(c)$, so $2v(a)=v(b)+v(c)$. This contradicts our assumption that $v(a)>v(b)\geq v(c)$.

Therefore, $R[[Y]]/R[[T]]$ cannot have branch partition $(2,1,1)$.
\end{proof}

Now we classify the Hurwitz trees for a $(\Z/2)^3$-cover of type $(4,4,4)$. Note that this is the smallest possible conductor triple of a $(\Z/2)^3$-cover, since $m_1\equiv -1$ mod $2^{3-1}$ by Theorem \ref{branchcyclecriterion}.

\begin{proposition}\label{444prop}
The only possible branch partition of a Hurwitz tree for a lift of a $(\Z/2)^3$-cover over $k$ of type $(4,4,4)$ is $(1,1,1,1,1,1,1)$, i.e. with equidistant geometry and $7$ branch points. See Figure \ref{444fig} below.
\end{proposition}

\begin{proof}
We study possible Hurwitz trees $T$ for a lift $\hat{C}$ of $C$, a $(\Z/2)^3$-cover over $R[[Z]]/R[[T]]$, by looking at subtrees corresponding to lifts of its $(\Z/2)^2$-subcovers. Let $C_1:R[[Z]]^{G_1}/R[[T]]$, $C_2:R[[Z]]^{G_2}/R[[T]]$, $C_3:R[[Z]]^{G_3}/R[[T]]$ be three generating $\Z/2$-subcovers of $\hat{C}$, in the notations of Definition \ref{type}. Below, I will use the same letter to indicate that several branch points belong to the same branch in the Hurwitz tree. For example, $a_i$ and $a_j$ are closer to each other than $a_i$ is to $b_k$.

Recall by Theorem \ref{44}, a subtree of $T$, corresponding to the Hurwitz tree of a $(\Z/2)^2$-cover of type $(4,4)$, can only have branch partition $(1,1,1,1,1,1)$, $(3,3)$ or $(2,2,2)$.

Case 1: Suppose $C_1\times C_2$ has Hurwitz tree with branch partition $(2,2,2)$, with branch points $a_1,a_2,b_1,b_2,c_1,c_2$. Without loss of generality, assume that $C_1,C_2$ have branch loci $\{a_1,a_2,b_1,b_2\}$ and $\{b_1,b_2,c_1,c_2\}$ respectively. Then by the branch cycle criterion (Theorem \ref{branchcyclecriterion}), without loss of generality, we can assume that $C_3$ has branch points $a_1,b_1,c_1$ and a new branch point $d$.

The Hurwitz tree of $C_3$ has at least two branches with only one branch point, and by Lemma \ref{211} it cannot have branch partition $(2,1,1)$, so it has to have equidistant geometry. Therefore $d$ is not on the same branch of $T$ as any branch point of $C_1\times C_2$, i.e. $\hat{C}$ has Hurwitz tree with branch partition $(2,2,2,1)$, see below. Then the subtree corresponding to $C_1\times C_3$ has branch locus $\{a_1,a_2,b_1,b_2,c_1,d\}$, thus has branch partition $(2,2,1,1)$, not an allowed Hurwitz tree for Klein-four covers by Theorem \ref{44}.

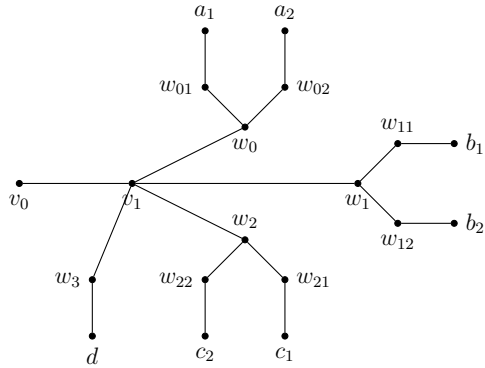
\begin{center}
    
\begin{tikzpicture}[node distance = 1cm, scale=0.8, every node/.style={scale=0.75}]
	\tikzstyle{main node}=[draw,circle,fill=black,minimum size=3pt,
	inner sep=0pt]
	\tikzstyle{empty node} = [draw, color=white, fill=white,opacity = 0]
	
	\node[main node] (v0) [label=below:$v_0$] {};
	\node[empty node] (emp) [right of=v0] {};
	\node[main node] (v1) [right of =emp,label=below:$v_1$] {};
	\node[empty node] (emp1) [right of=v1] {};
	\node[empty node] (emp2) [right of=emp1] {};
	\node[empty node] (emp3) [right of=emp2] {};
	\node[main node] (w0) [above of =emp2,label=below:$w_0$] {};
	\node[main node] (w01) [above left of =w0,label=left:$w_{01}$] {};
	\node[main node] (w02) [above right of =w0,label=right:$w_{02}$] {};
	\node[main node] (w1) [right of =emp3,label=below:$w_1$] {};
	\node[main node] (w11) [above right of =w1,label=above:$w_{11}$] {};
	\node[main node] (w12) [below right of =w1,label=below:$w_{12}$] {};
	\node[main node] (w2) [below of =emp2,label=above:$w_2$] {};
	\node[main node] (w21) [below right of =w2,label=right:$w_{21}$] {};
	\node[main node] (w22) [below left of =w2, label=left:$w_{22}$] {};
	\node[main node] (a) [above of = w01, label=above: $a_1$] {};
	\node[main node] (a') [above of = w02, label=above: $a_2$] {};
	\node[main node] (b) [right of = w11, label=right: $b_1$] {};
	\node[main node] (b') [right of = w12, label=right: $b_2$] {};
	\node[main node] (c) [below of = w21, label=below: $c_1$] {};
	\node[main node] (c') [below of = w22, label=below: $c_2$] {};
        \node[empty node] (emp4) [left of =w22] {};
        \node[main node] (w3) [left of =emp4,label=left: $w_3$]{};
        \node[main node] (d) [below of =w3, label=below: $d$]{};
	
	\draw (v0) -- (v1);
	\draw (v1) -- (w0);
	\draw (v1) -- (w1);
	\draw (v1) -- (w2);
	\draw (w0) -- (w01);
	\draw (w0) -- (w02); 
	\draw (w1) -- (w11);
	\draw (w1) -- (w12);
	\draw (w2) -- (w21);
	\draw (w2) -- (w22);
	\draw (w01) -- (a);
	\draw (w02) -- (a');
	\draw (w11) -- (b);
	\draw (w12) -- (b');
	\draw (w21) -- (c);
	\draw (w22) -- (c');
        \draw (v1) -- (w3);
        \draw (w3) -- (d);
	
\end{tikzpicture}
\end{center}

Case 2: Suppose $C_1\times C_2$ has Hurwitz tree with branch partition $(3,3)$, with branch points $a_1,a_2,a_3$ on the first branch, and $b_1,b_2,b_3$ on the second branch, where branch points on each branch are equidistant by Theorem \ref{44}. Without loss of generality, assume $C_1,C_2$ have branch loci $\{a_1,a_2,b_1,b_2\}$ and $\{a_1,a_3,b_1,b_3\}$ respectively. Then we can assume that $C_3$ has branch points $a_1,b_2,b_3$ and a new branch point $d$ by the branch cycle criterion.

Applying Lemma \ref{211} to $C_3$, $C_3$ has to have branch partition $(2,2)$, so $d$ must be on the same branch as $a_1,a_2,a_3$, i.e. $\hat{C}$ has Hurwitz tree $(4,3)$, see below. Then the third $\Z/2$-subcover $C_{12}$ of $C_1\times C_2$ has branch points $a_2,a_3,b_2,b_3$, and $C_{12}\times C_3$ has branch locus $\{a_1,a_2,a_3,d,b_2,b_3\}$. Thus the subtree corresponding to $C_{12}\times C_3$ is of shape $(4,2)$, not an allowed Hurwitz tree for Klein-four covers by Theorem \ref{44}.

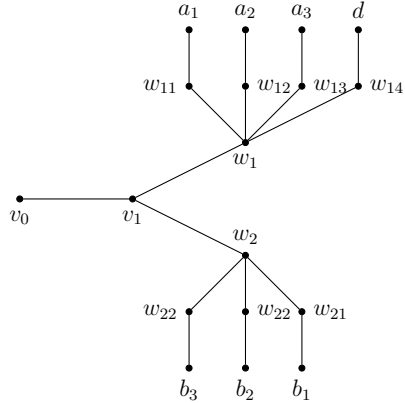
\begin{center}
\begin{tikzpicture}[node distance = 1cm, scale=0.8, every node/.style={scale=0.75}]
	\tikzstyle{main node}=[draw,circle,fill=black,minimum size=3pt,
	inner sep=0pt]
	\tikzstyle{empty node} = [draw, color=white, fill=white,opacity = 0]
	
	\node[main node] (v0) [label=below:$v_0$] {};
	\node[empty node] (emp) [right of=v0] {};
	\node[main node] (v1) [right of =emp,label=below:$v_1$] {};
	\node[empty node] (emp1) [right of=v1] {};
	\node[empty node] (emp2) [right of=emp1] {};
	\node[empty node] (emp3) [right of=emp2] {};
	\node[main node] (w1) [above of =emp2,label=below:$w_1$] {};
	\node[main node] (w12) [above of = w1, label= right: $w_{12}$] {};
	\node[main node] (w11) [left of =w12,label=left:$w_{11}$] {};
	\node[main node] (w13) [right of =w12,label=right:$w_{13}$] {};
	\node[main node] (w2) [below of =emp2,label=above:$w_2$] {};
	\node[main node] (w22) [below of = w2, label= right: $w_{22}$] {};
	\node[main node] (w21) [right of =w22,label=right:$w_{21}$] {};
	\node[main node] (w23) [left of =w22, label=left:$w_{22}$] {};
	\node[main node] (a) [above of = w11, label=above: $a_1$] {};
	\node[main node] (b) [above of =w12, label=above:$a_2$] {};
	\node[main node] (c) [above of = w13, label=above: $a_3$] {};
	\node[main node] (a') [below of = w21, label=below: $b_1$] {};
	\node[main node] (b') [below of =w22, label=below:$b_2$] {};
	\node[main node] (c') [below of = w23, label=below: $b_3$] {};
 \node[main node] (w14) [right of =w13, label=right: $w_{14}$]{};
 \node[main node] (d) [above of =w14, label=above: $d$]{};
	
	\draw (v0) -- (v1);
	\draw (v1) -- (w1);
	\draw (v1) -- (w2);
	\draw (w1) -- (w11);
	\draw (w1) -- (w12);
	\draw (w1) -- (w13) ;
        \draw (w1) -- (w14);
	
	\draw (w2) -- (w21);
	\draw (w2) -- (w22);
	\draw (w2) -- (w23);
	\draw (w11) -- (a);
	\draw (w12) -- (b);
	\draw (w13) -- (c);
	\draw (w21) -- (a');
	\draw (w22) -- (b'); 
	\draw (w23) -- (c'); 
	\draw (w14) -- (d);
	
\end{tikzpicture}
\end{center}

Case 3: Suppose $C_1\times C_2$ has Hurwitz tree with branch partition $(1,1,1,1,1,1)$. Without loss of generality, assume $C_1$ has branch locus $(a,a',b,b')$ and $C_2$ has branch locus $(b,b',c,c')$. Then by Theorem \ref{branchcyclecriterion}, we can assume $C_3$ has branch points $(a,b,c,d)$. If $d$ is on the same branch as one of the other points, say $a$, then the Hurwitz tree of $C_3$ would have branch partition $(2,1,1)$, which is a contradiction by Lemma \ref{211}. Therefore $d$ is on its own branch, and $T$ must have branch partition $(1,1,1,1,1,1,1)$, i.e. equidistant branch locus geometry. See Figure \ref{444fig} below.

Therefore, a $(\Z/2)^3$-cover over $k$ of type $(4,4,4)$ can only be lifted with equidistant geometry.
\end{proof}

\begin{figure}[h]
    \centering
    \begin{tikzpicture}[node distance = 1cm, scale=0.8, every node/.style={scale=0.9}]
	\tikzstyle{main node}=[draw,circle,fill=black,minimum size=3pt,
	inner sep=0pt]
	\tikzstyle{empty node} = [draw, color=white, fill=white,opacity = 0]
	
	\node[main node] (v0) [label=below:$v_0$] {};
	\node[empty node] (emp) [right of =v0] {};
	\node[main node] (v1) [right of =emp,label=below left :$v_1$] {};
	\node[empty node] (emp1) [right of = v1] {};
	\node[main node] (w4) [right of = emp1, label=above: $w_4$] {};
	\node[main node] (w3) [above of=w4, label=above: $w_3$] {};
	\node[main node] (w2) [above of=w3,label=above: $w_2$] {};
	\node[main node] (w1) [above of=w2,label=above: $w_1$] {};
	\node[main node] (w5) [below of=w4,label=below: $w_5$] {};
	\node[main node] (w6) [below of=w5,label=below: $w_6$] {};
	\node[main node] (w7) [below of=w6, label=below: $w_7$] {};
	\node[main node] (a) [right of= w1, label=right: $a$] {};
	\node[main node] (a') [right of= w2, label=right: $a'$] {};
	\node[main node] (b) [right of= w3, label=right: $b$] {};
	\node[main node] (b') [right of= w4, label=right: $b'$] {};
	\node[main node] (c) [right of= w5, label=right: $c$] {};
	\node[main node] (c') [right of= w6, label=right: $c'$] {};
 \node[main node] (d) [right of= w7, label=right: $d$] {};
	
	\draw (v0) -- (v1); 
	\draw (v1) -- (w1);
	\draw (v1) -- (w2);
	\draw (v1) -- (w3);
	\draw (v1) -- (w4);
	\draw (v1) -- (w5);
	\draw (v1) -- (w6);
        \draw (v1) -- (w7);
	\draw (w1) -- (a);
	\draw (w2) -- (a');
	\draw (w3) -- (b);
	\draw (w4) -- (b');
	\draw (w5) -- (c);
	\draw (w6) -- (c');
        \draw (w7) -- (d);
	
\end{tikzpicture}
    \caption{Equidistant Hurwitz tree for $(\Z/2)^3$-cover of type $(4,4,4)$}
    \label{444fig}
\end{figure}
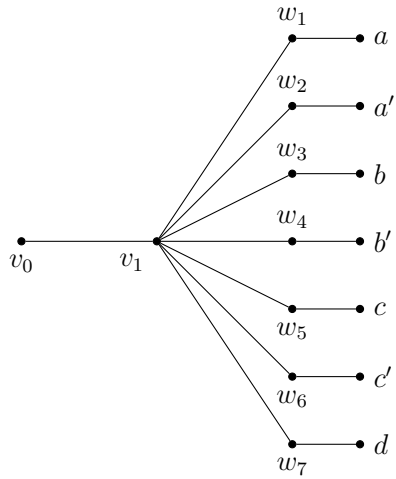

\begin{remark}
In this special case of $(\Z/2)^2$-cover of type $(4,4,4)$, the lift only has $7$ branch points. Since $(\Z/2)^3$ has $7$ Klein-four quotients, there are $7$ Klein-four subcovers, all with distinct branch loci. Therefore we can look at a candidate Hurwitz tree for the $(\Z/2)^3$-cover, take away one branch point at a time, and check if the remaining subtree is one of the allowed Klein-four Hurwitz trees. This method will allow us to reach the same conclusion. However, the above proof can be generalized to more general $(\Z/2)^3$-covers, if we know a classification of Klein-four Hurwitz trees with higher conductors.
\end{remark}

\subsection{Lifting $(\Z/2)^3$-Covers of Type $(4,4,2r),r\geq 3$}\label{442r}
In this section, I will construct lifts of any $(\Z/2)^3$-cover of type $(4,4,2r)$ for $r\geq 3$, using methods in Mitchell's thesis \cite{mitchell} and results in Pagot's thesis \cite{pagot}. The lifts have Hurwitz tree $(3,3,3,2,\ldots,2)$, with $r-3$ branches of size $2$. Define $\rho=2^\frac{1}{2r-1}$, and assume $\rho\in \pi R$ after possibly enlarging $R$.

\begin{lemma}\label{222lifting}
Let $\alpha\in k^*,\beta\in k, A\in R^*$, and suppose that $U\in R^*$ is any element such that $-AU^2\equiv \alpha$ mod $\pi$ and $U-A\in R^*$. Then after possibly enlarging $R$, there exists $V\in R^*$ such that the following property holds: Let $$T_1=0,\ T_2=\rho^{4r-4} A,\ T_3=\rho U,\ T_4=\rho U+\rho^{4r-4}V.$$
Then the cover $\displaystyle Y^2=F(T^{-1})=\prod_{i=1}^4(1-T_iT^{-1})$ of $\P^1_R$
has good reduction, namely with reduction $\displaystyle z^2-z=\frac{\alpha}{t^3}+\frac{\beta}{t}$.
\end{lemma}

\begin{proof}
This proof is similar to the proof of Lemma 4.2.2 in \cite{mitchell}, but with different and more general distances between branch points.

Let $V=-\rho^2B-A+(-\rho^3(\rho^2B+A)U)^{1/2}$, for some $B\in R$ with $B \equiv \beta$ mod $\pi$. Assume $V\in R$ after possibly enlarging $R$. Then $V$ is a solution to the polynomial equation 
\begin{align*}
V^2+2(\rho^2B+A)V+\rho^3(\rho^2B+A)U+(\rho^2B+A)^2&=0;\\
\text{or equivalently, } \rho^{4r-5}UV+ U^2-(\rho^{2r-2}B+\rho^{2r-4}A+U+\rho^{2r-4}V)^2&=0.
\end{align*}

Thus
\begin{equation}\label{defV}
(\rho^{4r-5}UV+ U^2)^{1/2}=-\rho^{2r-2}B-\rho^{2r-4}A-U-\rho^{2r-4}V,
\end{equation}
where $(\rho^{4r-5}UV+ U^2)^{1/2}$ denotes the appropriate square root of $\rho^{4r-5}UV+ U^2$. After possibly enlarging $R$, we can assume this element is in $R$, along with $V$.

For $r\in R$, let $o(r)$ denote a polynomial in $R[T^{-1}]$ with Gauss valuation strictly greater than $v(r)$, i.e. all the coefficients have valuations strictly greater than $v(r)$. Using the definitions of $T_i$ and $\rho$, we have that
\begin{align*}
F(T^{-1}) &= (1-\rho^{4r-4}AT^{-1})(1-\rho UT^{-1})(1-(\rho U+\rho^{4r-4}V)T^{-1})\\
&= 1-(\rho^{4r-4}A+2\rho U+\rho^{4r-4}V)T^{-1}+(\rho^{4r-3}UV+\rho^2 U^2)T^{-2}-4AU^2T^{-3}+o(4).
\end{align*}

Again after enlarging $R$, let $$q = (\rho^{4r-3}UV+\rho^2 U^2)^{1/2}=\rho(\rho^{4r-5}UV+ U^2)^{1/2}\pi R,$$and define $Q(T^{-1})=1+qT^{-1}\in R[T^{-1}]$.

Then by equation (\ref{defV}) and the definitions of $\rho$ and $q$,
\begin{align*}
&Q(T^{-1})^2+4BT^{-1}-4AU^2T^{-3}\\
=&Q(T^{-1})^2-2\rho((\rho^{4r-5}UV+ U^2)^{1/2}+\rho^{2r-4}A+U+\rho^{2r-4}V)T^{-1}-4AU^2T^{-3}\\
=&1+2qT^{-1}+q^2T^{-2}-2qT^{-1}-\rho^{2r}(\rho^{2r-4}A+U+\rho^{2r-4}V)T^{-1}-4AU^2T^{-3}\\
=&1-(\rho^{4r-4}A+2\rho U+\rho^{4r-4}V)T^{-1}+(\rho^{4r-3}UV+\rho^2 U^2)T^{-2}-4AU^2T^{-3}\\
=&F(T^{-1})+o(4).
\end{align*}

After the change of variables $Y=-2Z+Q(T^{-1})$, and using the above equality, the equation for the cover $Y^2=F(T^{-1})$ gives
\begin{align*}
4Z^2-4ZQ(T^{-1})+Q(T^{-1})^2
&=Q(T^{-1})^2+4BT^{-1}-4AU^2T^{-3}+o(4).\\
\text{Equivalently,}\quad Z^2-ZQ(T^{-1})&=BT^{-1}-AU^2T^{-3}+o(1).
\end{align*}
Finally, since $Q(T^{-1})\equiv 1$ mod $\pi$, by definitions of $A,B$ and $U$, this reduces to $z^2-z=\frac{\alpha}{t^3}+\frac{\beta}{t}$.
\end{proof}

\begin{proposition}\label{442rprop}
For all $(\Z/2)^3$-covers defined by a ring extension $k[[z]]/k[[t]]$ of type $(4,4,2r)$, $r\geq 3$, there exists a lift to characteristic $0$ with branch locus geometry $(3,3,3,\underbrace{2,\ldots,2}_{r-3})$. I.e. its Hurwitz tree has $3$ branches of size $3$ and $r-3$ branches of size $2$ (see Figure \ref{3332r}). In particular, the branch points of a lift here can never be equidistant.
\end{proposition}

\begin{figure}[h]
\centering
\begin{tikzpicture}[node distance = 1cm, scale=0.8, every node/.style={scale=0.9}]
	\tikzstyle{main node}=[draw,circle,fill=black,minimum size=3pt,
	inner sep=0pt]
	\tikzstyle{empty node} = [draw, color=white, fill=white,opacity = 0]
	
	\node[main node] (v0) [label=below:$v_0$] {};
	\node[empty node] (emp) [right of=v0] {};
	\node[empty node] (emp0) [right of =emp] {};
	\node[main node] (v1) [right of =emp0,label=below:$v_1$] {};
	\node[empty node] (emp1) [right of=v1] {};
	\node[empty node] (emp2) [right of=emp1] {};
	\node[empty node] (emp3) [right of=emp2] {};
	\node[main node] (w0) [above of =emp2,label=below:$w_0$] {};
	\node[empty node] (w02) [above of =w0] {};
	\node[main node] (w01) [left of =w02,label=left:$w_{01}$] {};
	\node[main node] (w03) [right of =w02,label=right:$w_{03}$] {};
	\node[main node] (w1) [right of =emp2,label=below:$w_1$] {};
	\node[empty node] (w12) [right of = w1] {};
	\node[main node] (w11) [above of =w12,label=above:$w_{11}$] {};
	\node[main node] (w13) [below of =w12, label=below:$w_{13}$] {};
	\node[main node] (w2) [below of =emp2,label=above:$w_2$] {};
	\node[empty node] (w22) [below of = w2] {};
	\node[main node] (w23) [right of =w22,label=right:$w_{23}$] {};
	\node[main node] (w21) [left of =w22, label=right:$w_{21}$] {};
	\node[main node] (w3) [below of =emp0, label=left:$w_3$] {};
	\node[main node] (w32) [below right of =w3,label=left:$w_{32}$] {};
	\node[main node] (w31) [below left of =w3, label=left:$w_{31}$] {};
	\node[main node] (a_0) [above of = w01, label=above: $U_1$] {};
	\node[main node] (b_0) [above of = w02, label=above: $U_1'$] {};
	\node[main node] (c_0) [above of = w03, label=above: $\tilde{U_1}$] {};
	\node[main node] (a_1) [right of = w11, label=right: $U_2$] {};
	\node[main node] (b_1) [right of = w12, label=right: $U_2'$] {};
	\node[main node] (c_1) [right of = w13, label=right: $\tilde{U_2}$] {};
	\node[main node] (a_2) [below of = w21, label=below: $U_3$] {};
	\node[main node] (b_2) [below of = w22, label=below: $U_3'$] {};
	\node[main node] (c_2) [below of = w23, label=below: $\tilde{U_3}$] {};
	\node[main node] (a_3) [below of = w31, label=below: $U_4$] {};
	\node[main node] (b_3) [below of = w32, label=below: $\tilde{U_4}$] {};
	\node[empty node] (w4) [left of =w31, label=below: $\cdots$] {};
	
	\draw (v0) -- (v1);
	\draw (v1) -- (w0);
	\draw (v1) -- (w1);
	\draw (v1) -- (w2);
	\draw (v1) -- (w3);
	\draw (w0) -- (w01);
	\draw (w0) -- (w03);
	\draw (w1) -- (w11);
	\draw (w1) -- (w13);
	\draw (w2) -- (w21);
	\draw (w2) -- (w23);
	\draw (w3) -- (w31);
	\draw (w3) -- (w32);
	\draw (w01) -- (a_0);
	\draw (w01) -- (b_0);
	\draw (w03) -- (c_0);
	\draw (w11) -- (a_1);
	\draw (w11) -- (b_1);
	\draw (w13) -- (c_1);
	\draw (w21) -- (a_2);
	\draw (w21) -- (b_2);
	\draw (w23) -- (c_2);
	\draw (w31) -- (a_3);
	\draw (w32) -- (b_3);
	
\end{tikzpicture}
\caption{Hurwitz tree with branch partition $(3,3,3,2,\ldots,2)$}
\label{3332r}
\end{figure}
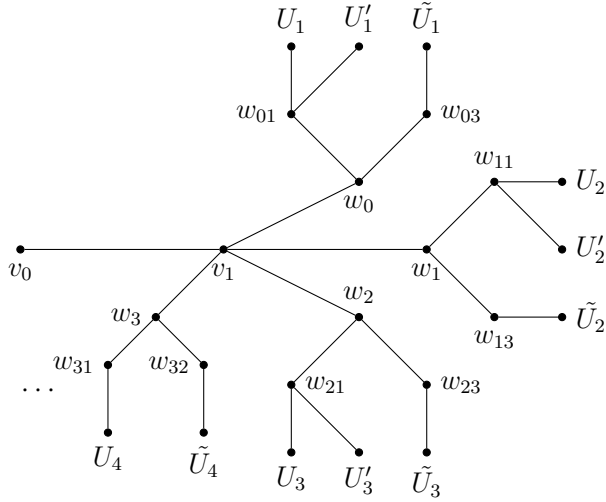

\begin{proof}
By Proposition \ref{naseq} and Remark \ref{1position}, we can assume that $k[[z]]/k[[t]]$ is defined as the composition of subcovers of the form
\begin{align*}
C_1:y_1^2-y_1&=\frac{a_1}{t^3}+\frac{b_1}{t},\\
C_2:y_2^2-y_2&=\frac{a_2}{t^3}+\frac{b_2}{t},\\
C_3:y_3^2-y_3&=\frac{1}{t^{2r-1}},
\end{align*}
where $a_1,a_2\neq 0$ are distinct. Pick a complete discrete valuation ring $R$ with residue field $k$, and enlarge $R$ if necessary. Fix $A\in R^*$ and $U_1,U_2\in R^*$ such that $-AU_i^2\equiv a_i$ mod $\pi$ and $U_i-A\in R^*$. Then by Lemma \ref{222lifting}, there exist $V_1,V_2\in R^*$, such that $$\mathcal{C}_i:Y_i^2=(1-\rho^{4r-4}AT^{-1})(1-\rho U_iT^{-1})(1-(\rho U_i+\rho^{4r-4}V_i)T^{-1})$$
is a lift of $C_i$ for $i=1,2$. Note that since $a_1\neq a_2$ and $A$ is a unit, $v(U_1-U_2)=0$.

Now let $T_1=0, T_2=U_1, T_3=U_2$, and choose $T_i, 4\leq i\leq r$, such that $v(T_i-T_j)=0$ for all $i\neq j$. Then by Lemma 5.1.2 of \cite{pagot} (see also Proposition 3.3 of \cite{onPagot}), we can define some $F(X)=\prod_{i=1}^r(X-T_i)(X-\tilde{T}_i)$ such that $v(T_i-\tilde{T}_i)=v(2)$, and $Y^2=F(X)$ has good reduction relative to the coordinate $T=\rho X$, with reduction $C_3$. Then $\tilde{T_i}=T_i+2W_i$ for some $W_i\in R^*$, and this lift $\mathcal{C}_3$ is defined by
$$Y_3^2:=((\rho/T)^rY)^2=\prod_{i=1}^r(1-\rho T_iT^{-1})(1-(\rho T_i+\rho^{2r}W_i)T^{-1}),$$
Observe that $0$ is the common branch point for all three lifts, while $\rho^{4r-4} A$ is a branch point that is shared by $\mathcal{C}_1,\mathcal{C}_2$; $\rho U_1$ is shared by $\mathcal{C}_1,\mathcal{C}_3$; and $\rho U_2$ is shared by $\mathcal{C}_2,\mathcal{C}_3$. Thus the lifts satisfy the branch cycle criterion (Theorem \ref{branchcyclecriterion}), and the normalization of the product of $\mathcal{C}_1,\mathcal{C}_2,\mathcal{C}_3$ is a lift of $k[[z]]/k[[t]]$. Let $T_1':=\rho^{4r-5}A$, $T_i':=U_i+\rho^{4r-5}V_i$ for $i=2,3$. It is straightforward to check that this configuration of branch
points is as indicated in Figure \ref{3332r}.
\end{proof}



\medskip

\noindent{\bf Author Information:}

\medskip
 
\noindent Jianing Yang\\
Department of Mathematics, University of Pennsylvania, Philadelphia, PA 19104-6395, USA\\
email: yangjianing1995@gmail.com, jianingy@sas.upenn.edu

\medskip

\noindent The research was supported in part by NSF grant DMS-2102987.

\end{document}